\documentclass{amsart}
\usepackage{MnSymbol}
\usepackage[mathscr]{euscript}

\usepackage{hyperref}
\hypersetup{
    colorlinks,
    citecolor=black,
    filecolor=black,
    linkcolor=black,
    urlcolor=black
}

\usepackage{comment}
\usepackage{stmaryrd} 

\DeclareFontFamily{OT1}{pzc}{}
\DeclareFontShape{OT1}{pzc}{m}{it}{<-> s * [1.100] pzcmi7t}{}
\DeclareMathAlphabet{\mathpzc}{OT1}{pzc}{m}{it}

    \usepackage{etoolbox}
    \patchcmd{\section}{\scshape}{\large\bfseries}{}{}
    \makeatletter
    \renewcommand{\@secnumfont}{\bfseries}
    \makeatother

\usepackage{tikz-cd}
\tikzcdset{arrow style=math font}
\tikzcdset{scale cd/.style={every label/.append style={scale=#1},
    cells={nodes={scale=#1}}}}

\definecolor{color1}{rgb}{0.9,0.9,0.9}
\definecolor{color2}{rgb}{0.5,0.5,0.5}

\usepackage[maxbibnames=99]{biblatex}
\usepackage{booktabs}
\bibliography{references}

\numberwithin{equation}{section}
\newtheorem{theorem}{Theorem}[section]
\newtheorem{corollary}[theorem]{Corollary}
\newtheorem{lemma}[theorem]{Lemma}
\newtheorem{proposition}[theorem]{Proposition}
\theoremstyle{definition}

\newtheorem{remark}[theorem]{Remark}
\newtheorem{example}[theorem]{Example}

\def\CC{\mathcal{C}}
\def\KK{\mathbb{K}}
\def\ZZ{\mathbb{Z}}

\def\VV{\mathcal{V}}
\def\lim{\mathsf{lim}}
\def\Pres{\mathsf{Pres}}
\def\epi{\twoheadrightarrow}
\def\mono{\rightarrowtail}

\def\Ker{\mathsf{Ker}}
\def\Im{\mathsf{Im}}
\def\CC{\mathsf{C}}
\def\NN{\mathsf{N}}
\def\DD{\mathsf{D}}
\def\DK{\mathsf{DK}}

\let\oldtocsection=\tocsection 
\let\oldtocsubsection=\tocsubsection 
\renewcommand{\tocsection}[2]{\hspace{0mm}\oldtocsection{#1}{#2}}
\renewcommand{\tocsubsection}[2]{\hspace{1em}\oldtocsubsection{#1}{#2}}

\begin{document}

\title{Limits via relations}

\begin{abstract}
In this paper, we study operations on functors in the category of abelian groups simplar to the derivation in the sense of Dold-Puppe. They are defined as derived limits of a functor applied to the relation subgroup over a category of free presentations of the group. The integral homology of the Eilenberg-Maclane space $K(\mathbb Z,3)$ appears as a part of description of these operations applied to symmetric powers.
\end{abstract}

\author{Sergei O. Ivanov} 
\email{ivanov.s.o.1986@gmail.com, ivanov.s.o.1986@bimsa.cn}
\address{
Beijing Yanqi Lake Institute of Applied Mathematics (BIMSA)}

\author{Roman Mikhailov} 
\address{Saint Petersburg University, 7/9 Universitetskaya nab., St. Petersburg, 199034 Russia}

\author{Fedor Pavutnitskiy}
\address{
Beijing Yanqi Lake Institute of Applied Mathematics (BIMSA)}

\thanks{The work of the second named author was performed at the Saint Petersburg Leonhard Euler International Mathematical Institute and supported by the Ministry of Science and Higher Education of the Russian Federation (agreement no. 075–15–2022–287).}

\maketitle

\section{Introduction}

Theory of limits and colimits in the category of presentations is studied in details in the series of papers. Speaking informally, theory of (derived) (co)limits is a way to design functors and natural transformations in algebraic categories with enough projective objects. 

In this paper, we introduce an operation on functors in the category of abelian groups, which looks like the derivation in the sense of Dold-Puppe. In particular, the short exact sequences of functors give rise to a long exact sequence of these new ``derived'' ones. 

For an abelian group $A,$ consider the category of free presentations ${\sf Pres}(A),$ with object being free groups $F$ with epimorphisms $F\to A.$ Morphisms are group homomorphisms over $A$. For any functor from the  category of presentations  $\mathcal F: {\sf Pres}(A)\to {\sf Ab}$
to the category of abelian groups, 
one can consider the derived limits ${\sf lim}^i\mathcal F,\ i\geq 0.$ These limits depend only on $A.$ Moreover, if we denote by ${\sf Pres}_{\sf Ab}$ the category, whose objects are epimorphisms from free abelian groups to abelian groups $F\epi A,$  and morphisms $(F\twoheadrightarrow A)\to (F'\twoheadrightarrow A')$ are commutative squares, then any functor $\mathcal F:{\sf Pres}\to {\sf Ab}$ defines a functor ${\sf Ab}\to {\sf Ab}$ 
given by $A\mapsto   \lim_{{\sf Pres}(A)} \mathcal F.$ 

Let $F$ be an endofunctor in the category of abelian groups. Consider an object of ${\sf Pres}(A),$  $0\to R\to F\to A\to 0,$  i.e. denote by $R$ the kernel of an epimorphism $F\to A.$ Our operation is the following:
$$
F\mapsto {\sf lim}^i FR_A:={\sf lim}^i F(R),\ i\geq 0. 
$$

Let's show how it works on the main examples. Denote by $\otimes^n, {\sf S}^n, \Lambda^n, \Gamma^n,\ n\geq 0,$
the tensor, symmetric, exterior and divided powers respectively. There are the following natural isomorphisms
\begin{align*}
& {\sf lim}^i \otimes^nR_A=L_{n-i}\otimes^n(A),\\ 
& {\sf lim}^i \Lambda^nR_A=L_{n-i}{\sf S}^n(A),\\ 
& {\sf lim}^i \Gamma^nR_A=L_{n-i}\Lambda^n(A),
\end{align*}
for $i=0,\dots, n.$ Here $L_i$ are derived functors in the sense of Dold-Puppe. The proof follows from the Koszul-type sequences and properties of limits, see \cite[Th. 8.1]{ivanov2015higher}.

The case of functors ${\sf lim}^i{\sf S}^nR_A,$
is more complicated and is the main subject of this paper. We show that, for a free abelian group $A,$ there is the following description 
\begin{align*}
& \lim^2 {\sf S}^2R_A=\tilde\otimes ^2(A)\\
& \lim^3 {\sf S}^3R_A=\tilde \otimes ^3(A)\\ 
& \lim^2 {\sf S}^3R_A= A\otimes \ZZ/3.
\end{align*}
with all other limits to be zero for ${\sf S}^2, {\sf S}^3.$ Here $\tilde\otimes^2, \tilde\otimes^3$ are anty-symmetric tensor square and cube respectively. 

The structure of functors ${\sf lim}^i{\sf S}^n(R)$ is complicated. We show that, for $A=\ZZ,$ there is an isomorphism
$$
H_n( K(\ZZ,3),\ZZ) \cong \bigoplus_{d\geq 2} \lim^{n-2d+1} S^dR_{\mathbb Z}.
$$
Here $H_n( K(\ZZ,3),\ZZ) $ is the $n$th integral homology group of the Eilenberg-MacLane space $K(\ZZ,3).$ This description follows from the following statement (see Corollary \ref{cor:non-natural_iso}): for a free finitely generated abelian group $A$ there are \textbf{non-natural} isomorphisms of abelian groups
$$
\lim^i {\sf S}^dR_A 
\cong
\begin{cases}
L_{i-1}\Gamma^d(A,1), & i<d, \\
L_{d-1}\Gamma^d(A,1) \oplus L_d \Gamma^d(A,1) , & i=d,\\
0,& i>d.
\end{cases}
$$
Here $L_i\Gamma^d(A,1)$ are derived functors in the sense of Dold-Puppe. The situation is interesting and rare. Usually when we have two complicated graded functors with the property that they are non-naturally isomorphic, there seems to be an either hidden natural isomorphism or a problem of non-splitting sequences raises in some way. However in this case, there is no natural isomorphism. The functorial description of the derived functors $L_i\Gamma^n(A,1)$ for free abelian groups $A$ and well as of the integral homology groups $H_i(K(A,3),\mathbb Z)$ is given in the paper \cite{breen2016derived}. All functors which appear in that description are known and there are no anty-symmetric powers.

The paper is organized as follows. In Section 2 we recall needed facts on Dold-Kan correspondence and cosimplicial groups. In Section 3, We show that, for a polynomial functor $\Phi$ of degree $d\geq 1,$ ${\sf lim}^i \Phi R_A=0$ for $i>d.$ Observe that the derived functors in the sense of Dold-Puppe have the same property. The key point of the proof is based on cosimplicial models constructed for arbitrary functors which give a way to compute derived limits. In Section 4 we recall Kuhn duality for functors and apply it in the context of limits. In particlular, we show how to describe the limits via shifted devided functors of the dual functors (see Corollary \ref{corintro}). In Section 5 we prove the mentioned above results on ${\sf lim}^i {\sf S}R_A.$

\section{Cosimplicial modules}

\subsection{Reminder on Dold–Kan correspondence}
In this subsection we remind some information of the Dold-Kan correspondence that can be found in \cite{weibel1995introduction}.

Let $\VV$ be an abelian category. For a simplicial object $V$ of an abelian category $\VV$ the non-normalised complex $\CC_\bullet V$ is a chain complex whose components are $\CC_n V=V_n$ and the differential is defined by the alternating sum of face maps. The normalised complex $\NN_\bullet V$ can be defined in two ways which are equivalent up to natural isomorphism: as a subcomplex of $\CC_\bullet V $ and as a quotient complex of $\CC_\bullet V$  \cite[Lemma 8.3.8]{weibel1995introduction}. For our purposes it is more convenient to define it as a quotient complex, whose components are 
\begin{equation}
\NN_n V = {\sf Coker}((s_0,\dots,s_{n-1}) : V_{n-1}^{\oplus n} \to  V_n).  
\end{equation}
The map $\CC_\bullet V\epi \NN_\bullet V$ is a chain homotopy equivalence \cite[Th.8.3.8]{weibel1995introduction}. If we denote by $\DD V$ the kernel of the map $\CC V\to \NN V,$ then there is a natural splitting of the short exact sequence $\DD V \mono \CC V \epi \NN V$
\begin{equation}
\CC V \cong \NN V \oplus \DD V,
\end{equation}
where $\DD V$ is a chain contractible complex. The homology of $\NN V$ and $\CC V$ is called homotopy groups of the simplicial object
\begin{equation}
\pi_* V = H_*(\NN_\bullet V)\cong H_*(\CC_\bullet V). 
\end{equation}

The construction of the normalised complex defines a equivalence between the category of simplicial objects and the category of non-negatively graded chain complexes, whose inverse functor is denoted by $\DK$
\begin{equation}
\NN_\bullet : \VV^{\Delta^{\sf op}} \leftrightarrows {\sf Ch}_{\geq 0}(\VV) : {\sf DK}_\bullet. 
\end{equation}

The functor $\DK_\bullet$ can be constructed as follows. For  a chain complex $U$ we define a simplicial object $\DK_\bullet U,$ whose components are
\begin{equation}
\DK_n U = \bigoplus_{\sigma:[n]\epi [k]} U_k,
\end{equation}
where the summation is taken by all surjective order-preserving maps $\sigma:[n]\epi [k],$ where $0\leq k\leq n$. If $f:[m]\to [n]$ is an order preserving map, then the map $f^*:\DK_n U\to \DK_m U$ is defined so that its component $(f^*)_{\tau,\sigma} : U_k \to U_l$ from the direct summand indexed by $\sigma:[n]\epi [k]$ to the direct summand indexed by $\tau:[m]\epi [l]$ has the following form
\begin{equation}
(f^*)_{\tau,\sigma} = 
\begin{cases} 
{\sf id}_{U_k},& \sigma f =\tau,  \\ 
\partial^U_k, & \sigma f = d^k \tau, \\ 
0, & \text{else.}
\end{cases}    
\end{equation}
Note that the construction of $\DK$ commutes with additive functors i.e. for any additive functor between abelian categories $\Phi:\VV\to \VV'$ there is an natural isomorphism 
\begin{equation}\label{eq:additive_DK}
\DK_\bullet (\Phi(U))\cong \Phi( \DK_\bullet(U)). 
\end{equation}

\subsection{Dold-Kan correspondence for cosimplicial modules}

Since the dual of an abelian category is also abelian, there is a dual version of this picture \cite[Cor./Def. 8.4.3]{weibel1995introduction}. The non-normalised cochain complex $\CC^\bullet V$ of a cosimplicial object $V$ is a cochain complex, whose components are $\CC^n V=V^n,$ and the differential is defined as the alternating sum of coface maps. The normalised cochain complex $\NN^\bullet V$ is a subcomplex of $\CC^\bullet V$ whose components are defined as 
\begin{equation}
\label{eq:ker_si}
\NN^n V = {\sf Ker}((s^0,\dots,s^{n-1})^T: V^n \to (V^{n-1})^{\oplus n+1} ).
\end{equation}
The monomorphiam $\NN^\bullet V \to \CC^\bullet V$ is a quasiisomorphism and cohomotopy groups of $V$ are defined as
\begin{equation}
\pi^*V=H^*(\NN^\bullet V)\cong H^*( \CC^\bullet V).
\end{equation}
The Dold-Kan correspondence is an equivalence between the category of cosimplicial objects and the category of non-negatively graded cochain complexes 
\begin{equation}
\NN^\bullet : \VV^\Delta \leftrightarrows {\sf CoCh}^{\geq 0}(\VV) : {\sf DK}^\bullet. 
\end{equation}

\begin{proposition}
Let $V$ be a cosimplicial module over a ring $\KK.$ Then the components of the normalized cochain complex $\NN V$ are intersections of kernels of codegeneracy maps 
\begin{equation}
\NN^n V = \bigcap_{i=0}^n {\sf Ker}(s^i:V^n \to V^{n-1}),
\end{equation}
and the differential on $\NN V$ is given by the restriction of the alternating sum of cofaces $\sum (-1)^id^i.$ 
\end{proposition}
\begin{proof}
It follows from the fact that $\NN^\bullet V$ is a subcomplex of $\CC^\bullet V$ and the formula \eqref{eq:ker_si}.
\end{proof}

\begin{proposition}\label{prop:DK(f)}
If we treat a $\KK$-homomorphism $f:U^0\to U^1$ as a cochain complex concentrated in degrees $0$ and $1,$ we obtain that the cosimplicial module $\DK^\bullet(f)$ has components
\begin{equation}
\DK^n(f) = (U^1)^{\oplus n}\oplus U^0.
\end{equation}
Its coface and codegeneracy maps are defined by
\begin{equation}
\begin{split}
d^i_{\DK^\bullet(f)}(x_0,\dots, x_{n-2}, y) &= \begin{cases} 
(0,x_0,\dots,x_{n-2},y), & i=0 
\\
(x_0,\dots, x_{i-1},x_{i-1},\dots,x_{n-2}, y), & 1 \leq i \leq n-1 
\\ 
(x_0,\dots, x_{n-1}, f(y), y), & i = n\end{cases} \\  
s^i_{\DK^\bullet(f)} (x_0,\dots, x_{n}, y) &= (x_0,\dots, \hat x_i,\dots,x_{n}, y).
\end{split}
\end{equation}
\end{proposition}
\begin{proof}
It is easy to check that the described cosimplicial module is well defined and that its normalized complex is $U^0\to U^1.$ 
\end{proof}

\section{Cochain complex of crossed effects computing higher limits}

\subsection{Cross effects of functors} 
Let $\KK$ be a ring and  $A_1,\dots,A_n$ is a collection of $\KK$-modules. For $1\leq j\leq n$ we denote by 
\begin{equation}
{\sf pr}^j : \bigoplus_{i=1}^n A_i \longrightarrow \bigoplus_{i\neq j} A_i, \hspace{1cm} {\sf em}^j : \bigoplus_{i\neq j} A_i \longrightarrow \bigoplus_{i=1}^n A_i
\end{equation}
the canonical projection and the canonical embedding. 

Now assume that  $\Phi: {\sf Mod}(\KK) \to {\sf Mod}(\KK)$ is a functor to an abelian category such that $\Phi(0)=0.$  For $n\geq 0$ the $n$-th crossed effect of $\Phi$ is a functor ${\sf Mod}(\KK)^n\to \mathcal {\sf Mod}(\KK)$ given by 
\begin{equation}
\Phi(A_1|\dots|A_n) = \Ker\left( \Phi\left(\bigoplus_{i=1}^n A_i\right) \longrightarrow \bigoplus_{i=1}^n \Phi\left(\bigoplus_{j\ne i} A_j \right)\right),
\end{equation}
where the homomorphisms are induced by the canonical projections.
Note that if $A_i=0$ for some $i,$ then the crossed effect vanishes 
\begin{equation}\label{eq:zero-cross}
\Phi (A_1|\dots| 0| \dots| A_n)=0.
\end{equation}
The cross effect is a direct summand of $\Phi(\bigoplus_{i=1}^n A_i)$ and there is a decomposition (see \cite[p.2]{drozd2003poly}, \cite[p.1149]{drozd2003cubic}, {\cite[p.18]{djament2022decompositions}})
\begin{equation}
\Phi(\bigoplus_{i=1}^n A_i) = \bigoplus_{s=0}^n \ \bigoplus_{1\leq j_1<\dots<j_s\leq n} \Phi(A_{j_1}|\dots|A_{j_s}).    
\end{equation}
A functor $\Phi$ is called polynomial (in the sense of Eilenberg-Mac Lane) of degree $\leq d$ if $\Phi_{d+1}=0.$ 

\begin{lemma}\label{lemma:cross_pr} 
There are equations  
\begin{equation}\label{eq:cross_pr1}
\begin{split}
\Ker(\Phi({\sf pr}^j)) &= \bigoplus_{s=1}^n \ \bigoplus_{j\in \{j_1<\dots<j_s\}} \Phi(A_{j_1}|\dots|A_{j_s})\\
\Im( \Phi( {\sf em}^j ) ) & = \bigoplus_{s=1}^n \ \bigoplus_{j\notin \{j_1<\dots<j_s\}} \Phi(A_{j_1}|\dots|A_{j_s})
\end{split}
\end{equation}
Moreover, we have
\begin{equation}\label{eq:cross_pr2}
\bigcap_{j=1}^{n-1} \Ker(\Phi({\sf pr}^j)) = \Phi(A_1|\dots|A_{n-1}) \oplus \Phi(A_1|\dots|A_n).
\end{equation}
\end{lemma}
\begin{proof}
We can think about the sum $\bigoplus_{i\ne j} A_i$ as about the sum $\bigoplus_{i=1}^n B_i,$ where $B_i=A_i$ for $i\ne j$ and $B_j=0.$ Then ${\sf pr}^j$ can be redefined as ${\sf pr}^j=1_{A_1 \oplus \dots \oplus A_{j-1}} \oplus 0_{A_i,0} \oplus 1_{A_{j+1}\oplus \dots \oplus A_n}$ and ${\sf em}^j$ can be redefined as ${\sf em}^j=1_{A_1 \oplus \dots \oplus A_{j-1}} \oplus 0_{0,A_i} \oplus 1_{A_{j+1}\oplus \dots \oplus A_n}.$
Then, using the functoriality of the crossed effect and \eqref{eq:zero-cross}, we obtain \eqref{eq:cross_pr1}. The formula \eqref{eq:cross_pr2} follows from \eqref{eq:cross_pr1}. 
\end{proof}

\subsection{A chain complex of crossed effects associated with a morphism} 

For a homomorphism $\varphi: B\to A $ and $0\leq i \leq n$ we consider the following maps
\begin{equation}
\Delta^{i,n}_\varphi : A^{\oplus n}\oplus B \longrightarrow A^{\oplus n+1} \oplus B
\end{equation}
defined by 
\begin{equation}
\Delta^{i,n}_\varphi(a_1,\dots, a_n,b) = 
\begin{cases}
(a_1,\dots,a_i,a_i,\dots,a_n,b), & i<n\\
(a_1,\dots, a_n, \varphi(b),b),  & i=n
\end{cases}
\end{equation}
We also use the following notation for the canonical projections and embeddings
\begin{equation}
{\sf pr}^{n+1}_\varphi  : A^{\oplus n}\oplus B
\longrightarrow 
A^{\oplus n},   \hspace{1cm}
{\sf em}^{n+1}_\varphi : 
A^{\oplus n} \longrightarrow A^{\oplus n} \oplus B,
\end{equation}
and set $\Delta^{i,n}_A:=\Delta^{i,n}_{1_A},$ ${\sf pr}^n_A:={\sf pr}^n_{1_A},$ ${\sf em}^n_A={\sf em}^n_{1_A}.$

\begin{lemma}\label{lemma:eq_nabla}
The following equations are satisfied 
\begin{enumerate}
\item Cosimplicial equations $
\Delta^{j,n+1}_\varphi \Delta^{i,n}_\varphi = \Delta^{i,n+1}_\varphi \Delta^{j-1,n}_\varphi,$ where $i<j;$

\item ${\sf pr}^{n+1}_\varphi \Delta^{i,n}_\varphi = \Delta^{i,n-1}_A {\sf pr}^n_\varphi,$ for $i\neq n;$

\item  ${\sf pr}^{n+1}_\varphi \Delta^{n,n}_\varphi = 1_{A^{\oplus n}}\oplus \varphi.$

\item  $\Delta^{i,n}_\varphi {\sf em}^n_\varphi = {\sf em}^{n+1}_\varphi \Delta^{i,n-1}_A,$ for $i\neq n;$

\item $\Delta^{n,n}_\varphi {\sf em}^n_\varphi = {\sf em}^{n+1}_\varphi {\sf em}^n_A  $

\item ${\sf pr}^{n+1}_\varphi \: \Delta_\varphi^{i,n} \: {\sf em}^n_\varphi = \Delta_A^{i,n-1}$ for $i\ne n;$
\item ${\sf pr}^{n+1}_\varphi \: \Delta_\varphi^{n,n} \: {\sf em}^n_\varphi = {\sf em}^n_A.$
\end{enumerate}
\end{lemma}
\begin{proof}
Direct computation. 
\end{proof}

One can say that we obtain a coaugmented semi-co-simplicial module
\begin{equation}
\begin{tikzcd}
B \ar[r] 
& 
A\oplus B 
\ar[r,shift left=1]
\ar[r,shift right=1]
& 
A^{\oplus 2}\oplus B 
\ar[r,shift left=2]
\ar[r]
\ar[r,shift right=2]
& 
\dots
\end{tikzcd}
\end{equation}
Further we set 
\begin{equation}
\Phi_{[n|1]}(A,B) = \Phi(A|\dots|A|B)
\end{equation}
and define morphisms $h^i_\varphi:\Phi_{[n|1]}(A,B) \to \Phi_{[n+1|1]}(A,B)$ as compositions
\begin{equation}
\begin{tikzcd}
\Phi_{[n|1]}(A,B) \ar[r,"h^i_\varphi"] \ar[d,rightarrowtail] & 
\Phi_{[n+1|1]}(A,B)  \\ 
\Phi(A^{\oplus n}\oplus V) \ar[r,"\Phi(\Delta^i_\varphi)"] & \Phi(A^{n+1}\oplus B) \ar[u,twoheadrightarrow]
\end{tikzcd}    
\end{equation}
Cosimplicial identities for
$\Delta^i_\varphi$ imply cosimplicial identities for $h^i_\varphi$: $h^j_\varphi h^i_\varphi = h^i_\varphi  h^{j-1}_\varphi$ where $i<j$.
So we can consider a cochain complex $C_\Phi(\varphi),$ whose components are
$C_\Phi(\varphi)^n = \Phi_{[n|1]}(A,B)$
and the differential is defined as the alternating sum of $h_i$
\begin{equation}
 C_\Phi(\varphi):\hspace{5mm} \Phi(B) \overset{h^0_\varphi}\longrightarrow \Phi(A|B) \overset{h^0_\varphi - h^1}\longrightarrow \Phi(A|A|B) \overset{h^0_\varphi-h^1_\varphi-h^2_\varphi}\longrightarrow \dots
\end{equation}
We will also use the notation $C_\Phi(A)=C_{\Phi}({\sf id}_A).$

It is easy to see that the construction of the complex is natural by the morphism $\varphi$ is natural by the morphism in the following sense. Any commutative square
\begin{equation}
\begin{tikzcd}
B \ar[r,"\alpha"] \ar[d,"\psi"] & A \ar[d,"\beta"]\\
B'\ar[r,"\psi"] & A'.    
\end{tikzcd}
\end{equation}
induces a morphism 
\begin{equation}
C_\Phi(\alpha,\beta) : C_\Phi(\varphi) \longrightarrow C_\Phi(\psi).
\end{equation}

\subsection{The standard complex for higher limits} For any object $c$ of any category with pairwise coproducts $\CC$ (not necessarily with an initial object) we denote by $c^{\sqcup n}$ the coproduct $\coprod_{i=0}^{n-1} c$ with embeddings $\alpha_i:c\to c^{\sqcup n}$ indexed by $ 0\leq i\leq n-1.$ Any morphism $f: c^{\sqcup n} \to d$ will be written as $f=(f_0,\dots,f_{n-1}),$ where $f_i=f\alpha_i.$ We consider the cosimplicial object $c^\bullet$ whose components are defined by 
\begin{equation}
(c^\bullet)^n = c^{\sqcup n+1}    
\end{equation}
the coface maps $d^i:c^{\sqcup n} \to c^{\sqcup n+1}$ and degeneracy maps $s^i:c^{\sqcup n+2} \to c^{\sqcup n+1}$  are defined by 
\begin{equation}\label{eq:ds}
d^i=(\alpha_0,\dots, \hat{\alpha}_i, \dots, \alpha_n), \hspace{1cm}
s^i=(\alpha_0,\dots,\alpha_i,\alpha_i,\dots,\alpha_n)
\end{equation}
for $0\leq i\leq n.$

\begin{theorem}[{\cite[Th. 2.12]{ivanov2020limits}}]\label{th:Fedor}
Let $\CC$ be a strongly connected category with pairwise coproducts, $\KK$ be a ring and $\Phi:\CC\to {\sf Mod}(\KK)$ be a functor. Then for any object $c$ of $\CC$ and any $i\geq 0$ we have an isomorphism 
\begin{equation}
\lim^i\: \Phi \cong \pi^i\Phi(c^\bullet).
\end{equation}
In particular, the right-hand side  of the isomorphism is independent of $c.$ 
\end{theorem}

Further we will use the notation for the chain complex
\begin{equation}
{\sf Lim}_{(c)} \Phi = N^\bullet (\Phi(c^\bullet)) 
\end{equation}

\begin{remark}
The Theorem \ref{th:Fedor} was proved in \cite[Th. 2.12]{ivanov2020limits} only for the case $\KK=\ZZ,$ but the proof for any $\KK$ can be done without any changes.
\end{remark}

\subsection{The relation functor}
Further we consider a $\KK$-module $A$ and the category of its presentations $\Pres(A).$ It is easy to check that this category is strongly connected and has pairwise coproducts $p\sqcup p': F\oplus F'\epi A$. This category is a full subcategory of the comma category ${\sf Mod}(\KK) \downarrow A$, whose objects are epimorphisms from free modules $p : F\epi A.$ The functor 
\begin{equation}
{\sf R}_A : \Pres(A) \longrightarrow {\sf Mod}(\KK), \hspace{1cm} {\sf R}_A(p) = {\sf Ker}(p)
\end{equation}
will be called relation functor. For a giver presentation $p:F\epi A$ we will always use the notation 
$R={\sf R}_A(p).$

\begin{proposition}\label{prop:relation_resolution}
For any presentation $p:F\epi A,$ if we denote by $\varphi_p:R\to F$ the embedding of the kernel, there is a natural isomorphism of cosimplicial modules
\begin{equation}
{\sf R}_A(p^\bullet) \cong \DK^\bullet(\varphi_p)
\end{equation}
(see Proposition \ref{prop:DK(f)}).
\end{proposition}
\begin{proof}
The coproduct in the category $\Pres(A)$ is defined by $p\sqcup p': F\oplus F' \to A.$ Therefore $p^{\sqcup n+1}: F^{\oplus n+1} \to A.$ The kernel of $p^{\sqcup n+1}$ consists of tuples $(a_0,\dots,a_{n})$ such that $\sum_{i=0}^n a_i\in R.$ Consider a map 
\begin{equation}
\theta_n : {\sf Ker}(p^{\sqcup n+1}) \longrightarrow F^{\oplus n}\oplus R,
\end{equation}
\begin{equation}
\theta_n(a_0,\dots,a_{n})=(a_0,\ a_0+a_1,\ a_0+a_1+a_2, \  \dots,\ \sum_{i=0}^n a_i).
\end{equation}
The homomorphism $\theta_n$ is an isomorphism with the inverse given by 
\begin{equation}
\theta_n^{-1}(a_0,\dots,a_{n-1},r)=(a_0,\ a_1-a_0, \ \dots,\ a_{n-1}-a_{n-2}, \ r-a_{n-1}).  
\end{equation}
By the definition \eqref{eq:ds} the cofaces and codeneracy maps of ${\sf R}_A(p^\bullet)$ are defined by the formulas
\begin{equation}
\begin{split}
d^i_{{\sf R}_A(p^\bullet)}(a_0,\dots,a_n) &= (a_0,\dots,a_{i-1},0,a_{i},\dots,a_n), \\ s^i_{{\sf R}_A(p^\bullet)}(a_0,\dots,a_n) &=(a_0,\dots,a_i+a_{i+1},\dots,a_n).
\end{split}
\end{equation}
A direct computation shows that $\theta_n$ respects the face and degeneracy maps. 
\end{proof}

\subsection{Limits of compositions with the relation functor}

Let $\KK$ be a commutative ring and 
\begin{equation}
\Phi:
{\sf Mod}(\KK) \longrightarrow 
{\sf Mod}(\KK)
\end{equation}
be a functor. For any $\KK$-module $A$ we are interested in the higher limits of the composition with the relation functor
\begin{equation}
\Phi {\sf R}_A : {\sf Pres}(A) \longrightarrow {\sf Mod}(\KK).
\end{equation}

\begin{theorem}\label{th:cross_effect_comlex} Let $\Phi:{\sf Mod}(\KK)\to {\sf Mod}(\KK)$ be a functor such that $\Phi(0)=0$ and $p:F\epi A$ be any presentation. Denote by $\varphi_p:R\to F$ the embedding of the kernel. Then there are natural isomorphisms
\begin{equation}
\begin{split}
{\sf Lim}_{(p)} \Phi {\sf R}_A &\cong {\sf Cone}(C_\Phi(1_F,\varphi_p):C_\Phi(\varphi) \to C_\Phi(F))[-1] . 
\end{split}
\end{equation}
\end{theorem}
\begin{proof} Set $N:={\sf Lim}_{(p)} \Phi {\sf R}_A.$
Since ${\sf R}^\bullet(p)^n=F^n\oplus R,$ and $s^i_{{\sf R}^\bullet(p)}={\sf pr}^i:F^n\oplus R\to  F^{n-1}\oplus R$ are projections, Lemma \ref{lemma:cross_pr} implies that components of the normalised complex can be described as
\begin{equation}
N^n = \Phi_{[n|1]}(F,R) \oplus \Phi_{[n]}(F).
\end{equation}
So now we need to describe the differentials $\partial^n_{N},$ which are restrictions of the alternating sum $\sum_{i=0}^{n-1} (-1)^id^i_{ \Phi( {\sf R}^\bullet(p) ) },$ and prove that they coincide with the differential for the cone. 

Further in this prove we will use the following notations $d^i=d^i_{{\sf R}^\bullet(p)};$ $d^i_\Phi=d^i_{\Phi({\sf R}^\bullet(p))} = \Phi(d^i_{{\sf R}^\bullet(p)});$
\begin{equation}
\pi: \Phi(F^{\oplus n}\oplus R ) \to \Phi_{[n|1]}(F,R), \hspace{1cm} \pi': \Phi(F^{\oplus n}\oplus R ) \to \Phi_{[n]}(F)
\end{equation}
are the projections; 
\begin{equation}
\rho : \Phi_{[n|1]}(F,R) \to \Phi(F^{\oplus n}\oplus R),
\hspace{1cm}
\rho'  : \Phi_{[n]}(F) \to \Phi(F^{\oplus n}\oplus R)
\end{equation}
are the embeddings. 
Therefore, $\partial^n$ can be described as the alternating sum of matrices
\begin{equation}
M_i\coloneq
\left(
\begin{matrix}
\pi d^i_\Phi \rho 
& 
\pi d^i_\Phi \rho'
\\ 
\pi' d^i_\Phi \rho
&
\pi' d^i_\Phi \rho'
\end{matrix}
\right)
\end{equation}
for $0\leq i\leq n+1.$
We will also consider the maps 
\begin{equation}
\tilde \pi: \Phi(F^{\oplus n}) \longrightarrow  \Phi_{[n]}(F), 
\hspace{1cm}
\tilde \rho: \Phi_{[n]}(F) \longrightarrow \Phi(F^{\oplus n}).
\end{equation}
and the maps 
\begin{equation}
{\sf em}^{n}_\varphi: F^{\oplus n} \longrightarrow F^{\oplus n}\oplus R, 
\hspace{1cm}   
{\sf pr}^{n}_\varphi: F^{\oplus n}\oplus R \longrightarrow F^{\oplus n}.
\end{equation}
Then we have
\begin{equation}
\pi' =  \tilde \pi \Phi({\sf pr}^{n}_\varphi), \hspace{1cm} \rho' =   \Phi({\sf em}^{n}_\varphi)\tilde \rho. 
\end{equation}

The map
$d^0$ is the embedding $F^n\oplus R \to F^{n+1}\oplus R,$ which is trivially mapped to the first summand. Therefore the image of $d^0_\Phi$ is the direct sum of all crossed effects, where the first summand does not appear (Lemma \ref{lemma:cross_pr}). Therefore $\pi d^0_\Phi=0$ and $\pi' d^0_\Phi=0,$ and hence 
$M_0=0.$

Further we will freely use the equations from Lemma \ref{lemma:eq_nabla} for the computations. Now assume that $0\leq i\leq n.$  Then $d^i = \Delta^{i-1}_\varphi.$ So we have 
\begin{equation}
\pi d^i_\Phi \rho = \pi \Phi(\Delta^{i-1}_\varphi) \rho  = h^{i-1}_\varphi,
\end{equation}
and
\begin{equation}
\pi' d^i_\Phi \rho = \tilde \pi  \Phi({\sf pr}^{n+1}_\varphi\: \Delta^{i-1}_\varphi) \rho = \tilde \pi  \Phi(\Delta^{i-1}_F\: {\sf pr}^n_\varphi ) \rho =0
\end{equation}
because by Lemma \ref{lemma:cross_pr} $\Phi({\sf pr}^n_\varphi) \rho=0.$ Further we have
\begin{equation}
\pi d^i_\Phi \rho' = \pi \Phi(\Delta_\varphi^{i-1}{\sf em}^n_\varphi) \tilde \rho =  \pi \Phi({\sf em}^{n+1}_\varphi \Delta_F^{i-1} ) \tilde \rho = 0
\end{equation}
because by Lemma \ref{lemma:cross_pr} 
$\pi\Phi({\sf em}_\varphi^{n+1})=0$.
Finally we have
\begin{equation}
\pi' d^i_\Phi \rho' =\tilde \pi \Phi( {\sf pr}^{n+1}_\varphi \Delta^{i-1}_\phi {\sf em}^{n}_\varphi ) \tilde \rho = \tilde \pi \Phi(\Delta^{i-1}_F) \tilde \rho = h^{i-1}_F.
\end{equation}
Therefore for $1\leq i\leq n$ we have 
\begin{equation}
M_i=
\left(
\begin{matrix}
h^{n}_\varphi 
& 
0
\\ 
0
&
h^{i-1}_F
\end{matrix}
\right).
\end{equation}

Now assume that $i=n+1.$ Then $d^{n+1}=\Delta^n_\varphi.$ Similarly to the previous case we obtain 
\begin{equation}
\pi d^i_\Phi \rho = h^n_\varphi.
\end{equation}
Further we have 
\begin{equation}
\pi' d^{n+1}_\Phi \rho =  \tilde\pi \Phi( {\sf pr}^{n+1}_\varphi \Delta^{n}_\varphi ) \rho =  \tilde\pi \Phi( 1\oplus \varphi ) \rho = \Phi_{[n|1]}(1,\varphi),
\end{equation}
and 
\begin{equation}
\pi d^{n+1}_\Phi \rho' = \pi \Phi( \Delta_\varphi^n {\sf em}^n_\varphi ) \tilde \rho 
= 
\pi \Phi({\sf em}^{n+1}_\varphi {\sf em}^n_F) \tilde \rho = 0,
\end{equation}
because $\pi \Phi({\sf em}^{n+1}_\varphi)=0.$ Finally we have 
\begin{equation}
\pi' d^{n+1}_\Phi \rho' 
=
\tilde \pi \Phi( {\sf pr}^{n+1}_\varphi \Delta^n_\varphi {\sf em}^n_\varphi ) 
\tilde \rho
= \tilde \pi \Phi( {\sf em}_F^n ) \tilde \rho =0,
\end{equation}
because $\tilde \pi \Phi( {\sf em}_F^n )=0.$ Therefore
\begin{equation}
M_{n+1} = 
\left(
\begin{matrix}
h^{n}_\varphi 
& 
0
\\ 
\Phi_{[n|1]}(1,\varphi)
&
0
\end{matrix}
\right).
\end{equation}
Therefore the differential will be given by 
\begin{equation}
\partial^n_N =
\left(
\begin{matrix}
 - \sum_{i=0}^{n} (-1)^i h^{i}_\varphi 
& 
0
\\ 
(-1)^{n+1} \Phi_{[n|1]}(1,\varphi)
&
- \sum_{i=0}^{n-1} (-1)^i h^{i}_F 
\end{matrix}
\right).
\end{equation}
It follows that
\begin{equation}
\partial^n_{N[1]} =
\left(
\begin{matrix}
\partial^{n+1}_{C_\Phi(\varphi)} 
& 
0
\\ 
(-1)^{n} \Phi_{[n+1|1]}(1,\varphi)
&
 \partial^{n}_{C_\Phi(F)} 
\end{matrix}
\right).
\end{equation}
Consider an automorphism  $\theta^n : N^n\to N^n$ defined by a matrix $ 
\left( \begin{smallmatrix}
 (-1)^n & 0 \\
 0 & 1
\end{smallmatrix} 
\right)
$ Then 
\begin{equation}
(\theta^{n+1})^{-1} 
\partial^n_{N[1]} 
\theta^n 
=
\left(
\begin{matrix}
-\partial^{n+1}_{C_\Phi(\varphi)} 
& 
0
\\ 
\Phi_{[n+1|1]}(1,\varphi)
&
 \partial^{n}_{C_\Phi(F)} 
\end{matrix}
\right).
\end{equation}
The right-hand side of the above formula is the formula for the differential of the cone. Therefore, $\theta^n$ defines an isomorphism between $N[1]$ and ${\sf Cone}(C_\Phi(1_F,\varphi):C_\Phi(\varphi) \to C_\Phi(F)).$ 
\end{proof}

\begin{corollary}\label{cor:Lim(id)} Under the assumption of Theorem \ref{th:cross_effect_comlex},
if $A$ is a free $\KK$-module, we obtain 
\begin{equation}
{\sf Lim}_{({\sf id})}\Phi {\sf R}_A \cong C_\Phi(A)[-1] 
\end{equation}
\end{corollary}

\begin{corollary}
If $\Phi$ is a polynomial functor of degree $d,$ then 
\begin{equation}
\lim^i \Phi{\sf R}_A = 0, \hspace{1cm} i>d. 
\end{equation}
\end{corollary}

\section{Kuhn duality and higher limits}

In this section we assume that $\KK$ is a principal ideal domain. In this case, for a finitely generated module $A$ and a finitely generated presentation $R\mono F \epi A$ the module $R$ is also free and finitely generated.

For a module $A$ we set $A^\vee={\sf Hom}_{\KK}(A,\KK).$ For a functor $\Phi:{\sf Mod}(\KK)\to {\sf Mod}(\KK)$ we denote by $\Phi^\# : {\sf Mod}(\KK)\to {\sf Mod}(\KK)$ the functor defined by the formula 
\begin{equation}
\Phi^\#(A) := \Phi(A^\vee)^\vee.
\end{equation}
The functor $\Phi^\#$ is known as Khun dual to the functor $\Phi.$

\begin{proposition}
Let $A$ be a finitely generated $\KK$-module, and $p:F\epi A$ be a finitely generated presentation, then there is an isomorphism 
\begin{equation}
   ( {\sf Lim}_{(p)} \Phi {\sf R}_A )^\vee \cong  N_\bullet (\Phi^\# ( {\sf DK}_\bullet( F^\vee \to R^\vee ) )). 
\end{equation}
If we also assume that $\Phi$ takes finitely generated free modules to finitely generated free modules, we obtain 
\begin{equation}\label{eq:lim_dual}
{\sf Lim}_{(p)} \Phi {\sf R}_A  \cong  (N_\bullet (\Phi^\# ( {\sf DK}_\bullet( F^\vee \to R^\vee ) )))^\vee. 
\end{equation}
\end{proposition}
\begin{proof}
Since ${\sf R}_A(p^\bullet)\cong {\sf DK}^\bullet(R\mono F)$ by 
Proposition \ref{prop:relation_resolution}, using \eqref{eq:additive_DK}  we obtain ${\sf R}^\bullet(p)^\vee={\sf DK}_\bullet(F^\vee \to  R^\vee).$ Since for a finitely generated module $F$ we have a natural isomorphism $(F^\vee)^\vee\cong F,$ we obtain
\begin{equation}
\Phi({\sf R}^\bullet(p))^\vee = \Phi^\#({\sf R}^\bullet(p)^\vee) =  \Phi^\# ( {\sf DK}_\bullet( F^\vee \to R^\vee ) ). 
\end{equation}
The first isomorphism of the statement follows from Proposition \ref{prop:relation_resolution}. The second isomorphism follows from the fact that, if $\Phi(F^{\oplus n})$ and $\Phi(F^{\oplus n}\oplus R)$ are free finitely generated, then the cross effects $\Phi_{[n]}(F),$ $\Phi_{[n|1]}(F,R)$ are also free finitely generated modules (because they are direct summands and $\KK$ is a principal ideal domain), and hence, by Theorem \ref{th:cross_effect_comlex}, components of ${\sf Lim}_{(p)}\Phi {\sf R}_A$ are also free finitely generated.   
\end{proof}
Further we set 
\begin{equation}
A^{\diamond} = {\sf Ext}^1_\KK(A,\KK). 
\end{equation}

\begin{corollary}\label{corintro}
Assume that $A$ is a free finitely generated $\KK$-module and $\Phi$ takes free finitely generated modules to free finitely generated modules. Then there is an isomorphism of chain complexes
\begin{equation}
 {\sf Lim}_{({\sf id})} \Phi {\sf R}_A \cong (N_\bullet (\Phi^\# ( K(A,1) )))^\vee 
\end{equation} 
and for each $i$ we have a short exact sequence 
\begin{equation}\label{eq:cor_lim_dual_ses}
0 \to (L_{i-1}\Phi^\#(A^\vee,1))^\diamond \to \lim^i\Phi{\sf R}_A \to (L_i\Phi^\#(A^\vee,1))^\vee \to 0  
\end{equation}
\end{corollary}
\begin{proof}
It follows from isomorphism \eqref{eq:lim_dual}, the fact that ${\sf DK}_\bullet(A^\vee\to 0)=K(A^\vee,1)$ and from the universal coefficient theorem for chain complexes.  
\end{proof}

\begin{corollary}
Assume that $A$ is a torsion finitely generated $\KK$-module and $\Phi$ takes free finitely generated modules to free finitely generated modules. Then for each $i$ there is a short exact sequence
\begin{equation}
0 \to (L_{i-1}\Phi^\#(A^\diamond ,0))^\diamond \to \lim^i\Phi{\sf R}_A \to (L_i\Phi^\#(A^\diamond,0))^\vee \to 0  
\end{equation}
\end{corollary}
\begin{proof}
Since $A$ is torsion module, $A^\vee=0.$ Therefore, $F^\vee\to R^\vee$ is a free resolution of $A^\diamond.$ It follows that  ${\sf DK}_\bullet(F^\vee \to R^\vee)$ is a free simplicial resolution of $K(A^\diamond,0),$ and hence $\pi_i (\Phi^\#({\sf DK}_\bullet(F^\vee \to R^\vee))) = L_i \Phi^\#(A^\diamond,0).$ Then the assertion follows from the isomorphism \eqref{eq:lim_dual} and the universal coefficient theorem for chain complexes. 
\end{proof}

\section{Higher limits for symmetric powers}

In this section we assume that $\KK=\ZZ$ and study $\lim^i S^d{\sf R}_A,$ where $S^d$ is $d$-th symmetric power and $A$ is a free finitely generated abelian group. In this description we will use results of \cite{breen2016derived} about the description of the non-additive derived functors of the functors of divided powers. Namely we will use that for a finitely generated free abelian group $A$ we have 
\begin{equation}\label{eq:L_i_Gamma}
L_i \Gamma^d (A,1) 
= 
\begin{cases}
\Lambda^d (A), & i=d,\\
\text{a finite abelian group}, & 0<i<d,\\
0, & i=0 \text{ or } i>d.
\end{cases}
\end{equation}
(see \cite[Theorem 6.3]{breen2016derived}, and the beginning of \S 7 in \cite{breen2016derived}). 

For an abelian group $A$ we denote by $\otimes(A)$ the tensor algebra. Then the symmetric algebra ${\sf S}(A)$ is its quotient by the relations $ab=ba$ for $a,b\in A,$ and the exterior algebra $\Lambda(A)$ is its quotient by the relations $a^2=0$ for $a\in A.$ We will also consider the anti-symmetric algebra $\tilde\otimes(A)$ which is a quotient of $\otimes(A)$ by the relations $ab=-ba$ for $a,b\in A.$ Since the relation $ab=-ba$ is satisfied in $\Lambda(A),$ we have a natural epimorphism $\tilde\otimes(A)\epi \Lambda(A).$ Consider the kernel
\begin{equation}
ASK(A):={\sf Ker}(\tilde\otimes(A)\epi \Lambda(A)).
\end{equation}
The ideal $ASK(A)$ is generated by elements $a^2$ for $a\in A.$ The graded abelian group $ASK(A)$ is a $2$-torsion group, because $a^2=-a^2$ in $\tilde\otimes(A).$ Since $\tilde\otimes(A)$ and $\Lambda(A)$ are naturally graded we can consider the homogeneous components $\Lambda^d(A), \tilde\otimes^d(A), ASK^d(A).$

Note that for a finite abelian group $T$ we have 
\begin{equation}
T^\diamond\cong {\sf Hom}(T,\mathbb{Q}/\ZZ)
\end{equation}
and the functor $(-)^\diamond$ induces a self-duality on the category of finite abelian groups, as well as the functor $(-)^\vee$ is a self-duality on the category of free finitely generated abelian groups. 

\begin{theorem}\label{th:limits_S^d} For a free finitely generated abelian group $A$ there is a natural isomorphism  
\begin{equation}\label{eq:iso_lim^i}
\lim^i S^d {\sf R}_A  \cong  (L_{i-1} \Gamma^d (A^\vee,1))^\diamond, \hspace{1cm} \text{for } i< d,
\end{equation}
a natural short exact sequence 
\begin{equation}\label{eq:ses_lim^d}
0 \longrightarrow (L_{d-1} \Gamma^d (A^\vee,1))^\diamond \longrightarrow \lim^d S^d {\sf R}_A \longrightarrow \Lambda^d(A) \longrightarrow 0
\end{equation}
and $\lim^i S^d{\sf R}_A=0$ for $i>d.$ Moreover, for a free abelian group $A$ (not necessarily finitely generated) there are natural isomorphisms 
\begin{equation}\label{eq:lim^dS^d=AS}
\lim^d S^d{\sf R}_A \cong \tilde\otimes^d(A),
\end{equation}
for $d\geq 1,$ and
$
\lim^1 S^d{\sf R}_A =0
$ for $d\geq 2.$
\end{theorem}
\begin{proof}
First we note that for a free finitely generated abelian group $A$ we have $\Gamma^d(A)=(S^d)^\#(A)$ and $\Lambda^d(A)=(\Lambda^d)^\#(A)$ (see \cite[\S 2.3]{breen2016derived}). 
The equation \eqref{eq:L_i_Gamma} implies that 
\begin{equation}
L_i \Gamma^d (A^\vee,1)^\vee 
= 
\begin{cases}
(\Lambda^d)^\# (A) = \Lambda^d(A), & i=d,\\
0, & i\neq d.
\end{cases}
\end{equation}
Then the isomorphism \ref{eq:iso_lim^i} and the short exact sequence \ref{eq:ses_lim^d} follow from the short exact sequence  \eqref{eq:cor_lim_dual_ses}.  

Let us prove the isomorphism \eqref{eq:lim^dS^d=AS}. Here we will use Corollary \ref{cor:Lim(id)}. Let's compute cross effects of $S^d$ and the maps $h_i$ for them. The symmetric algebra is an exponential functor $S(A\oplus B)=S(A)\otimes S(B),$ with the isomorphism defined by 
\begin{equation}
(a_1+b_1)\cdot {\dots} \cdot (a_n+b_n)\mapsto \sum_{\{i_1<\dots<i_k\} \sqcup \{j_1<\dots<j_l\}=\{1,\dots,n\}} (a_{i_1}\dots a_{i_k}) \otimes  (b_{j_1}\dots b_{j_l}).
\end{equation}
If we take $A=B$ and compose it with the diagonal map $\Delta:A\to A\oplus A,$ we obtain the map $S(A)\to S(A)\otimes S(A)$ defined by 
\begin{equation}\label{eq:diagonal}
a_1\cdot {\dots} \cdot a_n\mapsto \sum_{\{i_1<\dots<i_k\} \sqcup \{j_1<\dots<j_l\}=\{1,\dots,n\}} (a_{i_1}\dots a_{i_k}) \otimes  (a_{j_1}\dots a_{j_l}).
\end{equation}

We obtain $S(A_1\oplus \dots \oplus A_n)=S(A_1)\otimes \dots \otimes S(A_n)$ for any sequence of abelian groups $A_1,\dots, A_n.$ Therefore
\begin{equation}
S^d(A_1\oplus \dots \oplus A_n) = \bigoplus_{d_1+\dots+d_n=d, \ d\geq 0} S^{d_1}(A_1) \otimes \dots \otimes S^{d_n}(A_n).    
\end{equation}
It follows that 
\begin{equation}
S^d(A_1 | \dots | A_n) = \bigoplus_{d_1+\dots+d_n=d,\ d_i\geq 1} S^{d_1}(A_1) \otimes \dots \otimes S^{d_n}(A_n).
\end{equation}
In particular for $n=d-1,d$ and $A=A_1=\dots=A_d$ we obtain 
\begin{equation}
\begin{split}
S^d_{[d]}(A) &= A^{\otimes d}.   \\
S^d_{[d-1]}(A) & = \bigoplus_{i=1}^{d-1} A^{\otimes i-1}\otimes S^2(A) \otimes  A^{\otimes d-i-1}. 
\end{split}
\end{equation}
Now we need to compute $h_i:S^d_{[d-1]}(A)\to S^d_{[d]}(A).$ Looking on the formula \ref{eq:diagonal} we see that $h_i$ is trivial on all the summands $A^{\otimes j-1}\otimes S^2(A) \otimes A^{\otimes d-j-1}$ for $j\neq i$ and the restriction of of $h_i$ to $A^{\otimes i-1}\otimes S^2(A) \otimes A^{\otimes d-i-1}$ is given by the formula
\begin{equation}
\begin{split}
&h_i(a_1\otimes \dots \otimes   a_i a_{i+1} \otimes \dots \otimes a_d) = \\
&=a_1\otimes \dots  \otimes   a_i \otimes  a_{i+1} \otimes \dots \otimes a_d + a_1\otimes \dots  \otimes   a_{i+1} \otimes  a_{i} \otimes \dots \otimes a_d 
\end{split}
\end{equation}
It follows that the cokernel of the map $S^d_{[d-1]}(A) \to S^d_{[d]}(A)$ is the quotient of the tensor power  $A^{\otimes d}$ by the relations 
\begin{equation}
a_1\otimes \dots  \otimes   a_i \otimes  a_{i+1} \otimes \dots \otimes a_d = - a_1\otimes \dots  \otimes   a_{i+1} \otimes  a_{i} \otimes \dots \otimes a_d,
\end{equation}
which is equal to $\tilde\otimes^d(A).$ It follows that $\lim^d S^d{\sf R}_A \cong  H^{d-1}(C_{S^d}(A))\cong \tilde\otimes^d(A).$

For finitely generated free abelian groups $A$ the isomorphism \eqref{eq:iso_lim^i} implies 
$\lim^1 S^d{\sf R}_A \cong  L_0\Gamma^d(A^\vee,1)^\diamond =0.$ For all free abelian groups the isomorphism $\lim^1 S^d{\sf R}_A  =0$ follows from the fact that all the constructions $S^d(A), S^d_{[d]}(A), C_{S^d}(A), H^*(C_{S^d}(A))$ commute with filtered colimits of abelian groups, and that any free abelian group can be presented as a filtered colimits of its finitely generated subgroups, which are also free. 
\end{proof}

\begin{corollary}
For a free finitely generated abelian group $A$ there is a natural isomorphism 
\begin{equation}
L_{d-1}\Gamma^d(A,1) \cong  (ASK^d(A^\vee))^\diamond.
\end{equation}
\end{corollary}
\begin{proof}
The short exact sequence \ref{eq:ses_lim^d} implies that the torsion subgroup of $\lim^dS^d{\sf R}_A$ is naturally isomorphic to $(L_{d-1}\Gamma^d(A^\vee,1))^\diamond.$ The isomorphism \ref{eq:lim^dS^d=AS} implies that the torsion group of $\lim^dS^d{\sf R}_A$ is naturally isomorphic to $ASK^d(A).$ Therefore $(L_{d-1}\Gamma^d(A^\vee,1))^\diamond \cong ASK^d(A).$ Using that $(-)^\diamond$ is a self-duality for finite abelian groups and $(-)^\vee$ is a self-duality for free finitely generated abelian groups, we obtain the required isomorphism. 
\end{proof}

\begin{corollary}\label{cor:non-natural_iso}
For a free finitely generated abelian group $A$ there are \textbf{non-natural} isomorphisms of abelian groups
\begin{equation}
\lim^i S^d{\sf R}_A 
\cong
\begin{cases}
L_{i-1}\Gamma^d(A,1), & i<d, \\
L_{d-1}\Gamma^d(A,1) \oplus L_d \Gamma^d(A,1) , & i=d,\\
0,& i>d.
\end{cases}
\end{equation}
\end{corollary}
\begin{proof}
For any finitely generated free abelian group $A$ we have an non-natural isomorphism $A\cong A^\vee.$ For any finite group $T$ we have a non-natural isomorphism $T\cong {\sf Hom}(T,\mathbb{Q}/\ZZ)\cong T^\diamond.$  Therefore, using \eqref{eq:L_i_Gamma} for $i\ne d$ we have a non-natural isomorphism $L_i\Gamma^d(A^\vee,1)^\diamond \cong L_i\Gamma^d(A,1).$ Using \eqref{eq:iso_lim^i} we obtain the required isomorphism for $i<d.$
Since $L_d\Gamma^d(A,1)\cong \Lambda^d(A)$ is free abelian, we obtain that the short exact sequence \eqref{eq:ses_lim^d} non-naturally splits. The isomorphism for $i=d$ follows. 
\end{proof}

\begin{example}
Theorem \ref{th:limits_S^d} implies that for a free abelian group $A$ we have $\lim^1 S^2{\sf R}_{A}=0$ and $\lim^2S^2{\sf R}_A \cong \tilde\otimes^2(A).$ For $S^3$ we similarly have $\lim^1 S^3{\sf R}_A =0$ and $\lim^3 S^3{\sf R}_A = \tilde\otimes^3(A).$ We also claim that 
\begin{equation}
\lim^2 S^3{\sf R}_A \cong  A \otimes \ZZ/3.
\end{equation}
For a finitely generated free abelian group $A$ it follows from the fact that $L_1\Gamma^3(A,1)\cong A \otimes \ZZ/3$ (see \cite[(8-20)]{breen2016derived}) and that $(A^\vee \otimes \ZZ/3)^\diamond \cong A\otimes \ZZ/3.$ In order to prove it for arbitrary free abelian group $A$ we need to use that all the constructions commute with filtered colimits. A free abelian group $A$ can be presented as a filtered colimit of its finitely generated subgroups $A=\underset{B\subseteq_{\sf f.g.} A}{\sf colim} B,$ and using the fact that all the constructions that we use here commute with filtered colimits, we obtain $\lim^2 S^3{\sf R}_A \cong H^1( C_{S^3}(A)) \cong$ $  {\sf colim}_B H^1( C_{S^3}(B) )$ $\cong {\sf colim}_B B\otimes \ZZ/3 \cong A\otimes \ZZ/3.$
\end{example}

\begin{proposition}
For $n\geq 4$ we have an isomorphism of abelian groups
\begin{equation}
H_n( K(\ZZ,3),\ZZ) \cong \bigoplus_{d\geq 2} \lim^{n-2d+1} S^d{\sf R}_\ZZ
\end{equation}
{\rm (}here we assume that $\lim^i=0$ for $i<0${\rm )}.
\end{proposition}
\begin{proof}
In \cite[(B-3)]{breen2016derived}  (see also \cite[Satz 4.16]{dold1961homologie}) in is proved that for a free finitely generated abelian group $A$ and any $n\geq 0$ there is an isomorphism 
\begin{equation}
H_n(K(A,3),\ZZ) \cong \bigoplus_{d} L_{n-2d} \Gamma^d(A,1).
\end{equation}
Since $L_i\Gamma^d(A,1)=0$ for $i>d,$ we can assume that the summation is taken over all $d$ such that $n-2d\leq d.$ Equivalently we can rewrite this as $d\geq n/3.$  In particular, if $n\geq 4,$ we can assume that the summation is taken over indexes $d\geq 2.$ Since $L_d\Gamma^d(A,1)=\Lambda^d(A),$ we see that for $A=\ZZ$ and $d\geq 2$ we have $L_d\Gamma^d(\ZZ,1)=0.$ Therefore by Corollary \ref{cor:non-natural_iso} for $d\geq 2$ we have an isomorphism 
\begin{equation}
\lim^i S^d{\sf R}_\ZZ \cong L_{i-1}\Gamma^d(\ZZ,1).
\end{equation}
The assertion follows. 
\end{proof}

\printbibliography

\end{document}